\documentclass[oneside,english,reqno]{amsart}

\usepackage[latin1]{inputenc} 

\usepackage{amssymb}
\usepackage{epsf,epsfig,amsmath,amssymb,amsfonts,latexsym}
\usepackage{comment}
\usepackage{enumerate,enumitem}
\usepackage{subcaption}

\usepackage{color}
\usepackage{cancel}
\usepackage{amsmath}
\usepackage[normalem]{ulem}
\usepackage{amsfonts}
\usepackage{bbm}
\usepackage{amsthm}
\usepackage[mathscr]{eucal}
\usepackage{graphicx}

\usepackage{hyperref}
\hypersetup{
colorlinks = true, 
urlcolor = blue, 
linkcolor = red, 
citecolor = blue 
}
\usepackage{mathtools}
\usepackage{caption}
\usepackage{float}
\usepackage{tikz}
\usetikzlibrary{arrows,decorations.markings}
\usetikzlibrary{calc}

\def\N{\mathbb N}

\newcommand{\Z}{\mathbb{Z}}

\def \E{\mathbb{E}}

\def \m{\vec}

\def\mi{{\vec{i}}}
\def\mj{{\vec{j}}}

\newcommand{\Mod}[1]{\ (\text{mod}\ #1)}

\newcommand{\ZD}{\mathbb{Z}^d}

\newcommand{\cA}{\mathcal{A}}
\DeclareMathOperator{\dist}{dist}
\DeclareMathOperator{\degree}{deg}
\DeclareMathOperator{\lines}{Line}

\newcommand{\bo}{B}

\newtheorem{thm}{Theorem}[section]

\newtheorem{prop}[thm]{Proposition}
\newtheorem{cor}[thm]{Corollary}

\newtheorem{question}[thm]{Question}

\title{Mixing properties of colorings of the $\Z^d$ lattice}
\author{
Noga Alon
\and
Raimundo Brice\~no
\and
Nishant Chandgotia
\and
Alexander Magazinov
\and
Yinon Spinka
}

\address{Department of Mathematics, Princeton University, Princeton, NJ 08544, USA and Schools of Mathematics
and Computer Science, Tel Aviv University, Tel Aviv 6997801}
\email{nogaa@tau.ac.il}

\address{School of Mathematical Sciences, Tel Aviv University, Tel Aviv 69978, Israel}
\email{raimundo@alumni.ubc.ca}

\address{School of Mathematical Sciences, Hebrew University of Jerusalem, Israel}
\email{nishant.chandgotia@gmail.com}

\address{Higher School of Economics, National Research University, 6 Usacheva str., Moscow 119048, Russia}
\email{amagazinov@hse.ru}

\address{University of British Columbia, Department of Mathematics, Vancouver, BC V6T 1Z2, Canada}
\email{yinon@math.ubc.ca}

\begin{document}

\begin{abstract}
We study and classify proper $q$-colorings of the $\Z^d$ lattice, identifying three regimes where different combinatorial behavior holds: (1) When $q\le d+1$, there exist frozen colorings, that is, proper $q$-colorings of $\Z^d$ which cannot be modified on any finite subset. (2) We prove a strong list-coloring property which implies that, when $q\ge d+2$, any proper $q$-coloring of the boundary of a box of side length $n \ge d+2$ can be extended to a proper $q$-coloring of the entire box. (3) When $q\geq 2d+1$, the latter holds for any $n \ge 1$. Consequently, we classify the space of proper $q$-colorings of the $\Z^d$ lattice by their mixing properties.
\end{abstract}
\maketitle

\section{Introduction}
\label{section1}

A \emph{proper coloring} of a graph $G$ is an assignment of a color (say a number in~$\Z$) to each vertex of $G$ so that adjacent vertices are assigned different colors. For an integer $q \ge 2$, a (proper) \emph{$q$-coloring} of $G$ is a proper coloring in which all colors belong to a fixed set of size~$q$, e.g., $\{0,1,\dots,q-1\}$.

In this work, we mainly consider the $d$-dimensional integer lattice $\ZD$ for $d \ge 1$. We view it both as the group and its Cayley graph with respect to the standard generators. Notice that, in this case, the number of neighbors of each vertex is $2d$.

Our first result is about frozen $q$-colorings.
A $q$-coloring of $\Z^d$ is \emph{frozen} if any $q$-coloring of $\Z^d$ which differs from it on finitely many sites is identical to it. The existence of a frozen $q$-coloring precludes the possibility of any reasonable mixing property.


\begin{thm}
\label{thm:frozen-colorings}
There exist frozen $q$-colorings of $\Z^d$ if and only if $2 \le q \le d+1$.
\end{thm}

Since there are no frozen $q$-colorings of $\Z^d$ when $q \geq d+2$, a natural question is how unconstrained are proper colorings in this regime.
In order to understand better what happens when $q \geq d+2$, we show that a certain list-coloring property of large boxes in $\Z^d$ holds. A consequence of this property will be that whenever $q \geq d+2$, large boxes in $\Z^d$ have the property that any partial proper $q$-coloring of their boundary can be extended to a proper $q$-coloring of (the interior of) the entire box. To state the list-coloring result precisely, we now introduce some definitions.

For a graph $G$ and a function $L: G\to \N$, we say that $G$ is \emph{$L$-list-colorable} if for any collection of sets -- also called \emph{lists} -- $\{S_v\}_{v\in G}$ with $|S_v| \geq L(v)$, there exists a proper coloring $f$ of $G$ such that $f(v) \in S_v$ for all $v \in G$.

Denote $[n] := \{1,\dots,n\}$. Depending on the context, $[n]^d = \{1,\dots,n\}^d$ may also be interpreted as an induced subgraph of $\Z^d$.
Define $L_n^d: [n]^d \to \{2,\dots,d+2\}$ by
\[ L_n^d(\mi):=2 + |\{1\leq k \leq d~:~ 1<|i_k|< n\}|.\]

Our second result is the following.

\begin{thm}\label{thm:list-coloring}
The graph $[n]^d$ is $L_n^d$-list-colorable whenever $n \ge d+2$.
\end{thm}

The above results have implications for the mixing properties of the set of all $q$-colorings. Specifically, we characterise when the set of $q$-colorings of $\Z^d$ is topologically mixing, strongly irreducible, has the finite extension property and is topologically strong spatial mixing (see Section \ref{section4} for definitions and results). Mixing properties have important consequences in statistical physics~\cite{burtonsteiffnonuniquesft}, dynamical systems~\cite{MR2645044}, and the study of constrained satisfaction problems~\cite{briceno2019long}. We discuss some of these aspects in Section \ref{section5}.

\medskip
The rest of the paper is organized as follows.
Section~\ref{section2} is dedicated to frozen $q$-colorings and, in particular, contains the proof of Theorem~\ref{thm:frozen-colorings}. In addition, we also prove a result about non-existence of frozen $q$-colorings in general graphs satisfying appropriate expansion properties.
Section~\ref{section3} is dedicated to list-colorings and, in particular, contains the proof of Theorem~\ref{thm:list-coloring}. In Section~\ref{section4}, we introduce a hierarchy of mixing properties and, using the previous results, show that for a fixed dimension $d$, there exist two critical numbers of colors, namely $q = d+1$ and $q = 2d$, that determine three different mixing regimes.
Finally, in Section~\ref{section5}, we conclude with a discussion and open questions.

\medskip
We end this section with some notation that will be used throughout the paper.
The set of edges of a graph $G$ is denoted by $\E_G$. Given $U, V\subset G$, we denote by $\E(U,V) \subset \E_G$ the set of edges between a vertex of $U$ and a vertex of~$V$.
Given a set $U \subset G$, we denote the \emph{external vertex boundary} of $U$ by
\[ \partial U:=\{v \in G \setminus U~:~v\text{ is adjacent to some $u \in U$}\} .\]

%


\section{Frozen $q$-colorings}
\label{section2}

In this section, we prove Theorem~\ref{thm:frozen-colorings}. We split the proof into two parts: existence and non-existence of frozen $q$-colorings. Theorem~\ref{thm:frozen-colorings} is a direct consequence of Proposition~\ref{prop:frozen-colorings} and Corollary~\ref{cor:frozen-colorings} below. We begin with the existence of frozen $q$-colorings when $q \le d+1$. Many constructions have appeared in the past which are similar in principle (see, e.g., \cite[Section 8]{brightwell2000gibbs}).

\begin{prop}\label{prop:frozen-colorings}
There exist frozen $q$-colorings of $\Z^d$ for any $2 \le q \le d+1$.
\end{prop}


\begin{proof}
First suppose that $q = d+1$ and let $x\in \{0,1, \ldots, q-1\}^{\Z^d}$ be given by 
\[ x_{\mi}:= \sum_{k=1}^d ki_k \Mod{q}  \qquad\text{for all } \mi \in \ZD .\]

First, let us verify that $x$ is a $q$-coloring. Indeed, for $\mi \in \ZD$ and $\m e_k$ the unit vector in the $k$th direction ($1 \le k \le d$), we have that
$$x_{\mi+\m e_k}= x_{\mi}+k \Mod{q} .$$
In particular, $x_{\mi+\m e_k} - x_{\mi} = k \neq 0 \Mod{q}$, and thus adjacent vertices have different colors.

Let us now show that $x$ is frozen. Suppose that $y$ is a $q$-coloring that differs from $x$ on finitely many sites.
Among all $\mi$ such that $x_{\mi}\neq y_{\mi}$, we choose one which maximizes $\sum_{k=1}^d i_k$. Then, for any $1 \le k \le d$,
$$y_{\mi +\m e_k}=x_{\mi+\m e_k}= x_{\mi}+k \Mod{q} . $$
Therefore,
$$
\{y_{\mi +\m e_k}~:~ 1\leq k\leq d\}=\{0,1, \ldots, q-1\}\setminus \{x_\mi\} .$$
Since $y$ is a proper $q$-coloring, it must be that $y_{\mi}=x_\mi$, contradicting the choice of~$\mi$.

Now, to deal with the case $q < d+1$, notice that by the previous construction, we already have a frozen $q$-coloring of $\Z^{q-1}$. Thus, it suffices to prove that we can always extend a frozen $q$-coloring $x$ of $\Z^r$ to $\Z^{r+1}$, as we may then proceed by induction on $r$.  Given a frozen $q$-coloring $x$ of $\Z^r$, consider the $q$-coloring $y$ of $\Z^{r+1}$ defined as
$$y_{(i_1,\dots,i_{r+1})} := x_{(i_1,\dots,i_{r-1},i_r + i_{r+1})},$$
which is clearly proper, since if $(i_1,\dots,i_{r+1})$ is adjacent to $(j_1,\dots,j_{r+1})$ in $\Z^{r+1}$, then $(i_1,\dots,i_{r-1},i_r+i_{r+1})$ is adjacent to $(j_1,\dots,j_{r-1},j_r+j_{r+1})$ in $\Z^r$. Notice that $y$ is also frozen, since $y$ restricted to $\Z^r \times \{i_{r+1}\}$ can be seen as a translate of the frozen $q$-coloring $x$.
\end{proof}

Next we prove that no frozen $q$-coloring exists when $q \ge d+2$. While this follows from the list-coloring result given in Theorem~\ref{thm:list-coloring} (see also Theorem \ref{thm:mixing}), we give a direct argument here which applies in greater generality. Let us now state the result precisely.

The \emph{edge-isoperimetric constant} of a graph $G$ is
\[ h(G) := \inf_F \frac{|\E(F, G \setminus F)|}{|F|} ,\]
where the infimum is taken over all non-empty finite subsets of vertices $F$.
As for $\Z^d$, a $q$-coloring of $G$ is \emph{frozen} if any $q$-coloring of $G$ which differs from it on finitely many sites is identical to it.

\begin{prop}\label{prop:no-frozen-colorings}
Let $G$ be a graph of maximum degree $\Delta$ and $q > \frac12 \Delta + \frac12 h(G) + 1$. Then there do not exist frozen $q$-colorings of $G$.
\end{prop}

Since, in the $\Z^d$ case, $\Delta = 2d$ and $h(\Z^d) = 0$ (e.g., consider sets of the form $F = [n]^d$ for increasing $n$), we have the following corollary.

\begin{cor}\label{cor:frozen-colorings}
There do not exist frozen $q$-colorings of $\Z^d$ for any $q \geq d+2$.
\end{cor}

Proposition~\ref{prop:no-frozen-colorings} is sharp in many cases, including $\Z^d$ (Theorem~\ref{thm:frozen-colorings}), the triangular lattice (see Figure~\ref{figure:frozen}), the honeycomb lattice (trivially, since a $2$-coloring of an infinite connected bipartite graph is always frozen), and regular trees (see \cite{brightwell2002random,jonasson2002uniqueness}).

\begin{figure}[t]
\includegraphics[width=0.85\textwidth]{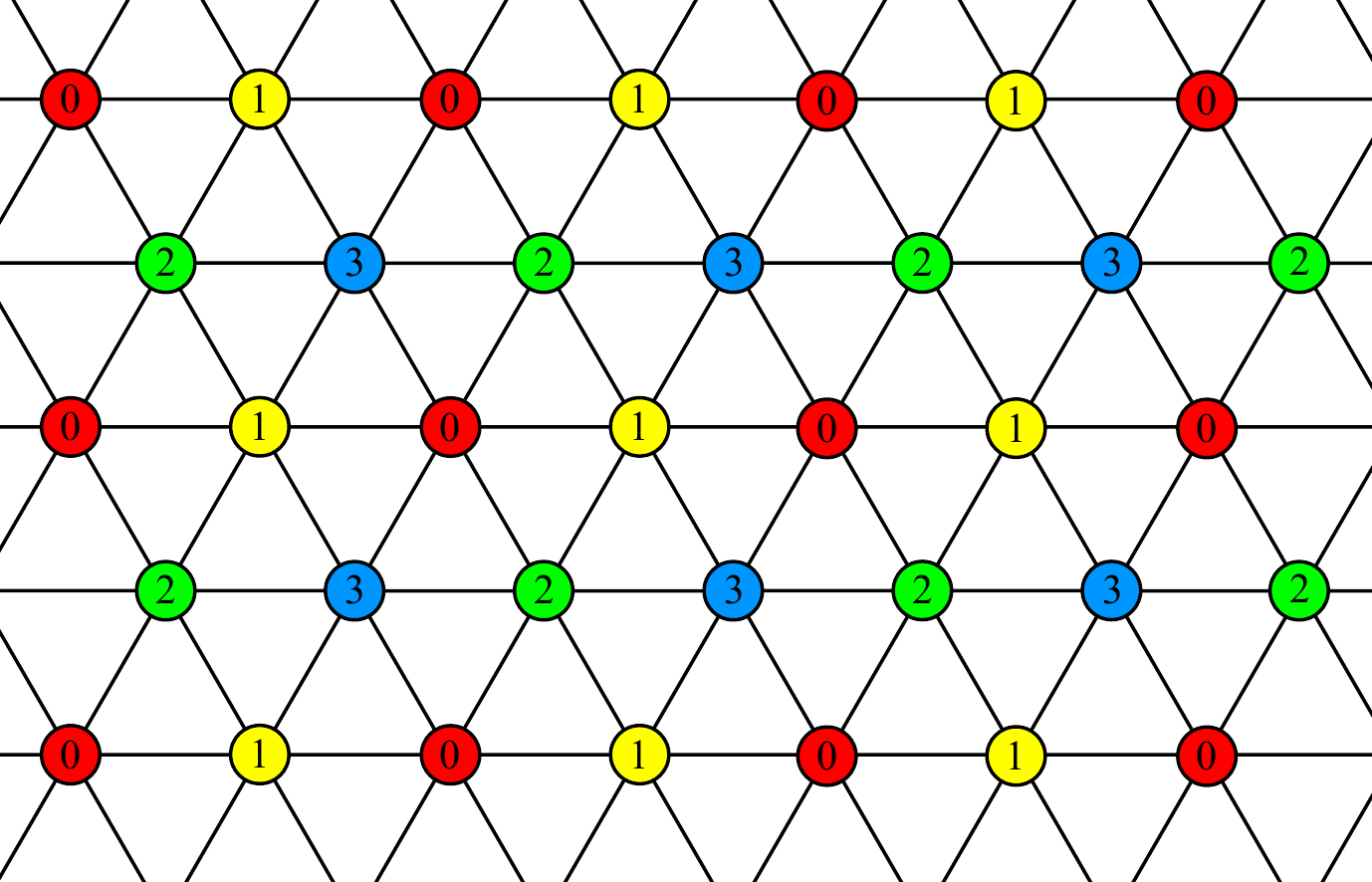}
\caption{A frozen $4$-coloring of the triangular lattice.}
\label{figure:frozen}
\end{figure}

We remark that finite (non-empty) graphs never admit frozen $q$-colorings since the colors can always be permuted. Proposition~\ref{prop:no-frozen-colorings} is thus trivial for finite graphs. Nevertheless, our argument can provide meaningful information for finite graphs as well. Proposition~\ref{prop:no-frozen-colorings} is an immediate consequence of the following.

Given a subset $F$ of vertices of $G$, we say that a $q$-coloring of $G$ is \emph{frozen on $F$} if any $q$-coloring of $G$ which differs from it on a subset of $F$ is identical to it. Thus, a $q$-coloring is frozen if and only if it is frozen on every finite set.


\begin{prop}\label{prop:no-frozen-colorings2}
Let $G$ be a graph, $q \ge 2$, and $F \subset G$ finite.
If $(q - 1) |F| > |\E_F| + |\E(F, G\setminus F)|$, then no $q$-coloring of $G$ is frozen on $F$.
\end{prop}

Since $2|\E_F| + |\E(F, G\setminus F)|$ equals the sum of the degrees of vertices in $F$, the following is an immediate corollary to Proposition~\ref{prop:no-frozen-colorings2}.

\begin{cor}\label{cor:no-frozen-colorings2}
Let $G$ be a graph of maximum degree $\Delta$, $q \ge 2$, and $F \subset G$ finite. If $(q - 1 - \frac \Delta 2) |F| > \frac12 |\E(F, G\setminus F)|$, then no $q$-coloring of $G$ is frozen on $F$.
\end{cor}

\begin{proof}[Proof of Proposition~\ref{prop:no-frozen-colorings2}]
Suppose towards a contradiction that there exists a $q$-coloring of $G$ which is frozen on $F$.
Consider the restriction of the coloring to the finite graph $G' := (F\cup \partial F, \E_F \cup \E(F, G\setminus F))$. 
For distinct colors $i$ and~$j$, let $G'_{i,j}$ be the subgraph of $G'$ consisting of edges between vertices colored $i$ and~$j$.
A \emph{bi-color component} is a connected component of any such $G'_{i,j}$.
Let $\cA$ be the collection of bi-color components.
Note that the bi-color components partition the edges of $G'$ so that $\sum_{A \in \cA} |\E_A| = |\E_F| + |\E(F, G\setminus F)|$.
Note also that, since the $q$-coloring is frozen on $F$, each $A \in \cA$ contains a vertex in $\partial F$, as otherwise, the two colors in $A$ could be swapped (contradicting that the coloring is frozen). Hence, each $A \in \cA$ contains at most $|\E_A|$ vertices of $F$. Thus, we have
\[ \sum_{A \in \cA} |A \cap F| \le \sum_{A \in \cA} |\E_A| = |\E_F| + |\E(F, G\setminus F)| .\]

On the other hand, using again that the $q$-coloring is frozen on $F$ (and hence also on each individual vertex in $F$), we see that each vertex in $F$ is contained in exactly $q-1$ bi-color components, and it follows that $\sum_{A \in \cA} |A \cap F| = (q-1) |F|$. We thus conclude that $(q - 1) |F| \le |\E_F| + |\E(F, G\setminus F)|$, contradicting the assumption.
\end{proof}

We remark that the proof of Proposition~\ref{prop:no-frozen-colorings2} shows something stronger, namely, that (under the assumption) any $q$-coloring of $G$ can be modified by swapping the two colors of a bi-color component (a so-called \emph{Kempe chain move}) contained in $F$.

We end this section with a short discussion about the existence of \emph{single-site frozen} $q$-colorings, those which are frozen on every $F$ having $|F|=1$.
Given a graph $G$ of maximum degree $\Delta$, it is clear that there do not exist single-site frozen $q$-colorings of $G$ whenever $q \ge \Delta + 2$. On the other hand, on $\Z^d$, it is straightforward to check that
$$x_{\mi} := \sum_{k=1}^d ki_k \Mod{2d+1}, \qquad \mi \in \Z^d ,$$
defines a single-site frozen $(2d+1)$-coloring of $\Z^d$. Similar constructions (like in the proof of Proposition \ref{prop:frozen-colorings}) yield single-site frozen $q$-colorings of $\Z^d$ for any $2 \le q \le 2d$. Thus, single-site frozen $q$-colorings of $\Z^d$ exist if and only if $2 \le q \le 2d+1$.

\section{List-colorability}
\label{section3}


In this section, we prove Theorem~\ref{thm:list-coloring}.
We will use the main result from~\cite{alon1992colorings}.
We say that a digraph $D$ is $L$-list-colorable if the underlying undirected graph is such. Define $L_D(v) := d^+_D(v)+1$, where $d^+_D(v)$ is the out-degree of $v$ in $D$.
The following is~\cite[Theorem~1.1]{alon1992colorings} specialized to digraphs having no odd directed cycles (this special case also follows from Richardson's Theorem; see~\cite[Remark~2.4]{alon1992colorings}).

\begin{thm}
\label{thm:digraph-list-coloring}
A finite digraph $D$ having no odd directed cycles is $L_D$-list-colorable.
\end{thm}

Thus, Theorem~\ref{thm:digraph-list-coloring} allows to prove $L$-list-colorability of an undirected graph $G$ by exhibiting an orientation of the edges of $G$ so that the out-degree of any vertex $v$ is strictly less than $L(v)$. The following provides a sufficient (and in fact necessary) condition for such an orientation to exist (see the closely related~\cite[Lemma~3.1]{alon1992colorings}).

\begin{cor}\label{cor:list-coloring}
Let $G$ be a finite bipartite graph and let $L: G\to \N$ satisfy
\[ \sum_{v\in H} (L(v)-1) \geq  |\E_H| \qquad\text{for any induced subgraph } H\subset G .\]

Then $G$ is $L$-list-colorable.
\end{cor}
\begin{proof}
Denote $G=(V,E)$ and consider the set $\mathcal{V} := \{ (v,i) : v \in V, 1 \le i \le L(v)-1 \}$ containing $L(v)-1$ copies of any vertex $v$.
Let $\mathcal{G}$ be the bipartite graph with bipartition classes $E$ and $\mathcal{V}$ in which $\{e,(v,i)\}$ is an edge in $\mathcal{G}$ if and only if $v$ is incident to $e$. For any $F \subset E$, letting $H$ be the subgraph of $G$ induced by the endpoints of $F$, we see that the number of neighbors of $F$ is $\sum_{v \in H} (L(v)-1) \ge |\E_H| \ge |F|$.
Thus, by Hall's theorem, $\mathcal{G}$ contains a matching of size $|E|$. Given such a matching, we obtain a digraph $D$ on $V$ by orienting each edge $e \in E$ from the vertex it is matched to outwards. In this orientation of $G$, the out-degree of any vertex $v$ is at most $L(v)-1$ so that $L_D(v) \le L(v)$. Since $G$ is bipartite, $D$ has no odd directed cycles. It follows from Theorem~\ref{thm:digraph-list-coloring} that $D$ (and hence also $G$) is $L$-list-colorable.
\end{proof}
\begin{proof}[Proof of Theorem~\ref{thm:list-coloring}]
Let $\{K_t\}_{t=0}^d$ be the level sets of $L_n^d$, that is, 
\begin{align*}
 K_{t}
  :&= \big\{\mi \in [n]^d~:~ L_n^d(\mi)=t+2\big\} \\
   &= \big\{ \mi \in [n]^d ~:~ |\{1\leq j \leq d~:~ 1<|i_j|<n\}|=t\big\}.
\end{align*}

Let $H$ be an induced subgraph of $G$. Denote $H_t:=H\cap K_t$ for $0\leq t\leq d$ and $H_{-1}:=\emptyset$. By Corollary~\ref{cor:list-coloring}, it suffices to show that
\begin{equation*}
\sum_{t=0}^d (t+1) |H_t|\geq \sum_{t=0}^d (|\E_{H_t}| + |\E(H_t, H_{t-1})|).
\end{equation*}
We will in fact prove this inequality term-by-term, namely, that for any $0 \leq t\leq d$,
\begin{equation}\label{eq:what_we_prove}
 (t+1) |H_t|\geq |\E_{H_t}| + |\E(H_t, H_{t-1})|.
\end{equation}
Since this is trivial for $t=0$ (the right-hand side is zero), we fix $1 \le t \le d$ and aim to show that~\eqref{eq:what_we_prove} holds for this $t$.


Write $\degree_v(G)$ for the degree of a vertex $v$ in a graph~$G$. Note that for all $\mi\in K_t$, we have that $\degree_\mi(K_t\cup K_{t-1})=2t$. Thus,
\begin{eqnarray*}
2t|H_t|=\sum_{\mi\in H_t} \degree_\mi(K_t\cup K_{t-1}).
\end{eqnarray*}

The right-hand side counts the sum of the number of oriented edges $(u,v)$ such that $u \in H_t$ and $v \in K_t\cup K_{t-1}$. We thus see that
\begin{eqnarray*}
2t|H_t|=2|\E_{H_t}|+|\E(H_t, K_t\setminus H_{t})|+ |\E(H_{t}, H_{t-1})|+ |\E(H_t, K_{t-1}\setminus H_{t-1})|. 
\end{eqnarray*}

Hence, to obtain~\eqref{eq:what_we_prove}, it suffices to show that
\begin{equation*}
2|H_t|+|\E(H_t, K_t\setminus H_{t})|+ |\E(H_t, K_{t-1}\setminus H_{t-1})|\geq |\E(H_t, H_{t-1})|.
\end{equation*}

Since putting $H_{t-1}=K_{t-1}$ only increases the right-hand side and decreases the left-hand side, it suffices to prove that
\begin{equation*}\label{equation:main_equation_got}
2|H_t|+|\E(H_t, K_t\setminus H_{t})|\geq |\E(H_t, K_{t-1})|.
\end{equation*}

To prove this, we partition $\E(H_t, K_{t-1})$ into two sets, $E_1$ and $E_2$, and show that $|E_1|\leq |\E(H_t, K_t\setminus H_{t})|$ and $|E_2|\leq 2|H_t|$.

Given $e \in \E(H_t, K_{t-1})$, letting $\mi \in H_t$ and unit vector $\m u$ be such that $e=\{\mi,\mi+ \m u\}$, we denote the line in the box $[n]^d$ in the direction $e$ by
\begin{align*}
 \lines(e)
  :=&~\{\mi-m \m u \in K_t~:~m \in \Z \} \\
  =&~\{\mi-m \m u~:~m=0,\dots,n-3\} .
\end{align*}

We now partition $\E(H_t, K_{t-1})$ into the two sets
\begin{align*}
 E_1 &:=\{e\in \E(H_t, K_{t-1})~:~\lines(e)\not\subset H_t \},\\
 E_2 &:=\{e\in \E(H_t, K_{t-1})~:~\lines(e)\subset H_t \}.
\end{align*} 

We begin by showing that $|E_1|\leq |\E(H_t, K_t\setminus H_{t})|$. To this end, it suffices to construct an injective map $f: E_1\to \E(H_t, K_t\setminus H_{t})$. For $e \in E_1$, define
\[ f(e):= \{ \mi- (m-1)\m u, \mi-m \m u \} ,\]
where $\mi \in H_t$ and unit vector $\m u$ are such that $e=\{\mi,\mi+\m u\}$, and where $m$ is the smallest positive integer such that $\mi-m \m u \in K_t\setminus H_t$. It is straightforward to check that $f$ is injective.

We now show that $|E_2|\leq 2|H_t|$. Denote the set of lines contained in $H_t$ by
\[ \mathcal{L}_t:=\{\lines(e) ~:~ e \in E_2\} , \]
and note that $|E_2|=2|\mathcal{L}_t|$.
Observe also that every vertex belongs to at most $d$ lines and that each line in $\mathcal{L}_t$ consists of exactly $n-2$ vertices of $H_t$. Thus,
\[ |\mathcal{L}_t|\leq \frac{d |H_t|}{n-2} \le |H_t| . \qedhere \]
\end{proof}

Let us make a remark about two aspects of the tightness of Theorem~\ref{thm:list-coloring}.
The first concerns the assumption that $n \ge d+2$: We show in Section \ref{section5} that $[2]^2$ is $L_2^2$-list-colorable (here $2=n<d+2=4$), while $[2]^3$ is not $L_2^3$-list-colorable (here $2=n<d+2=5$). Fixing $d$, the question of the `optimal' value of $n$ such that $[n]^d$ is $L_n^d$-list-colorable remains; the use of the word optimal is justified in the next proposition. The second aspect concerns the sizes of the lists: $[n]^d$ may be $L$-list-colorable for functions $L$ which are pointwise smaller-or-equal than $L^d_n$. For example, $[n]^2$ is $\min\{L^2_n,3\}$-list-colorable; see Question~\ref{q:L-list-colorability}.

\begin{prop}
\label{prop:mono}
Let $n \ge 2$ and $d \ge 1$. Suppose $[n]^d$ is $L_n^d$-list-colorable. Then,
\begin{enumerate}
\item $[n]^{d-1}$ is $L_n^{d-1}$-list-colorable and
\item $[n+1]^d$ is $L_{n+1}^d$-list-colorable.
\end{enumerate}
\end{prop}

\begin{proof}
Notice that the function $L_n^{d-1}$ is just the restriction $L_n^d|_{[n]^{d-1}\times \{1\}}$ after identifying $[n]^{d-1}$ with $[n]^{d-1}\times \{1\}$. This proves the first part. 

For the second part, we start by coloring $[n+1]^d\setminus [n]^d$ and then using the given hypothesis to color the portion which is left, that is, $[n]^d$. In order to color $[n+1]^d\setminus [n]^d$, we proceed by decomposing it into copies of $[n]^k$, where $0\leq k\leq d-1$, and successively coloring them in increasing order of $k$ (using the first part repeatedly).
\end{proof}

\section{Mixing properties}
\label{section4}

In this section, we study various aspects of rigidity and mixing of $X^d_q$, the set of all proper $q$-colorings of $\Z^d$ (the color set here is always taken to be $\{1,\dots,q\}$).
Let us discuss some consequences of the above theorems in terms of \emph{mixing properties}. The space $X^d_q$ is an instance of a so-called \emph{shift space}~\cite{MR2645044}, and the following properties, which we define only for $X^d_q$, are applicable in this more general context (for a general introduction to mixing properties, we would refer the reader to~\cite{MR2645044} and in the context of graph homomorphisms to~\cite{MR3702862,MR3743365}).

Given nonempty sets $U, V \subset \ZD$, we denote
$$\dist(U,V):=\min_{\mi \in U, \mj \in V}|\mi-\mj|_1.$$

Also, we denote by $B_n^d$ the set $\{-n,\dots,n\}^d$.

Four important mixing properties (in increasing order of strength) are:

\begin{enumerate}
\item $X^d_q$ is \emph{topologically mixing (TM)} if for all $U, V \subset \ZD$, there exists $n \in \N$ such that for all $\mi\in \ZD$ for which $d(U+\mi, V)\geq n$ and $x, y\in X^d_q$ there exists $z\in X^d_q$ such that $z|_{U+\mi}=x|_{U+\mi}$ and $z|_V = y|_V$.

\item $X^d_q$ is \emph{strongly irreducible (SI) with gap $n$} if for all $x, y\in X^d_q$ and $U, V \subset \ZD$ for which $\dist(U,V)\geq n$, there exists $z\in X^d_q$ such that $z|_U = x|_V$ and $z|_U = y|_V$.

\item $X^d_q$ has the \emph{finite extension property (FEP) with distance $n$} if for any $U \subseteq \ZD$ and any coloring $u$ of $U$, if $u$ can be extended to a $q$-coloring of $U + \bo_n^d$, then $u$ can be extended to a $q$-coloring of $\Z^d$.

\item $X^d_q$ is \emph{topologically strong spatial mixing (TSSM) with gap $n$} if for all $x, y\in X^d_q$ and $U, V, W \subset \ZD$ for which $\dist(U, V)\geq n$ and $x|_W = y|_W$, there exists $z\in X^d_q$ such that $z|_{U \cup W}=x|_{U \cup W}$ and $z|_{V \cup W}=y|_{V \cup W}$.
\end{enumerate}

The following implications hold:
\begin{equation}\label{eq:mixing-props}
\text{(TSSM)} \implies \text{(FEP)} \implies \text{(SI)} \implies \text{(TM)}.
\end{equation}
The first implication follows from \cite[Proposition 2.12]{2018factoring} and the additional observation that though the FEP property in \cite{2018factoring} is seemingly different from ours,  it is in fact equivalent (the gap might be different for the two definitions though). The other two implications are straightforward to verify.

A \emph{partial $q$-coloring of $U$} is a $q$-coloring of a subset $C$ of $U$, i.e., an assignment of colors in $\{1,\dots,q\}$ to each vertex in $C$ such that any pair of adjacent vertices in $C$ have different colors. We call $C$ the \emph{support} of the partial $q$-coloring.

We say that $X^d_q$ is \emph{$n$-fillable} if any partial $q$-coloring of $\partial [n]^d$ with support $C$ can be extended to a $q$-coloring of $[n]^d \cup C$. Notice that $X^d_q$ is $1$-fillable if and only if given any $q$-coloring of the neighbors of a vertex, we can always extend it to the vertex itself, and this is true if and only if $q \geq 2d+1$. In~\cite{2015integral}, this last property was called \emph{single-site fillability (SSF)} and used in the context of {shift spaces}.


\begin{prop}
\label{prop:list_coloring}
If $[n]^d$ is $L_n^d$-list-colorable, then $X^d_q$ is $n$-fillable for $q \geq d+2$. 
\end{prop}

\begin{proof} 
Fix $q\geq d+2$ and suppose that $[n]^d$ is $L_n^d$-list-colorable. Let $c$ be a partial $q$-coloring of $\partial [n]^d$ with support $C$. Now consider the lists
$S:\bo_n^d\to 2^{\{1, 2, \ldots, q\}}$ given by 
$$S(\mi):=\{1,2 \ldots, q\}\setminus \{c_{\mj}~:~\mj\in C\cap \partial \{\mi\}\}.$$

Then, clearly $|S(\mi)|\geq L_n^d(\mi)$. Since $[n]^d$ is $L_n^d$-list-colorable, we have that $c$ extends to a proper coloring of $[n]^d \cup C$.
\end{proof}

 For $d=2$ and $q \geq 4$, $n$-fillability already followed from \cite[Section 4.4]{MR1359979} and the analogous property for a $2 \times 2$ box is also true (see \cite{MR3819997}). 

For $q\leq d+1$ and any $n$, we can also construct an explicit partial $q$-coloring of $\partial [n]^d$ which cannot be extended to a $q$-coloring of $[n]^d$. Observe that the frozen $q$-coloring $x$ defined in the proof of Proposition \ref{prop:frozen-colorings} has an additional property: Restricting $x$ to the elements of $\partial [n]^d$ with at least one coordinate zero extends to a $q$-coloring of $[n]^d$ in a unique manner. Now we let 
$$C:=\{(i_1, \ldots, i_d)\in\partial [n]^d~:~i_k=0\text{ for some }k \}\cup \{(n+1,1, \ldots, 1)\}$$
and define $c$ to be the $q$-coloring of $C$ which coincides with $x$ everywhere except at $(n+1,1,\ldots, 1)$, where it is equal to $x_{(n,1,\ldots,1)}$; it is clear that $c$ cannot be extended to a coloring of $[n]^d$.
 
 Now we consider the following generalization to more general shapes.

\begin{prop}
\label{prop:fillable}
If $X^d_q$ is $n$-fillable, then for any subset $U \subseteq \Z^d$ that can be written as a union of translations of $[n]^d$ and any partial $q$-coloring $c$ of $\partial U$, there exists an extension of $c$ to a $q$-coloring of $U \cup C$, where $C$ is the support of $c$.
\end{prop}

\begin{proof}
Suppose that $X^d_q$ is $n$-fillable for some $n \in \N$. Consider any set $F \subset \Z^d$ such that $U = \bigcup_{\vec{j} \in F} (\vec{j} + [n]^d)$. Without loss of generality, suppose that $F$ is minimal, i.e., for every proper subset $F'$, $\bigcup_{\vec{j} \in F'} (\vec{j} + [n]^d)$ is a proper subset of $U$.

Now, given an arbitrary $\vec{i} \in F$, proceed by using the $n$-fillability property to find a coloring $a$ of $\vec{i} + [n]^d$ and considering its boundary to be colored according to $c$ restricted to $C \cap \partial(\vec{i} + [n]^d)$. Next, iterate this process with $U'$, $C'$, and $c'$, where $U'$ is $\bigcup_{\vec{j} \in F \setminus \{\vec{i}\}} (\vec{j} + [n]^d)$, $C'$ is $\partial U'$, and $c'$ is the restriction to $C'$ of the concatenation of the partial colorings $a$ and $c$. If $F$ is finite, this process finishes in at most $|F|$ steps, and if $F$ is infinite, we can conclude by the compactness of $\{1,\dots,q\}^{U}$.
\end{proof}

\begin{prop}
\label{prop:fillable_to_FEP}
If $X^d_q$ is $(2n+1)$-fillable, then it satisfies FEP with distance $2n$.
\end{prop}

\begin{proof}In this proof, we suppress the $d$ in the notation of $\bo_n^d$. Consider $U \subseteq \ZD$ and a coloring $u$ of $U$ that can be extended to a proper $q$-coloring $\overline{u}$ of $U + \bo_{2n}$. 

Now, consider a family of vectors $\{\vec{k}_\ell\}_{\ell \in \N}$ in $\Z^d$ such that $\{\vec{k}_\ell + B_n\}_{\ell \in \N}$ is a partition of $\Z^d$. For example, $\{n\vec{k} + B_n\}_{\vec{k} \in \Z^d}$ would be enough. Let $S \subseteq \N$ be the set of indices $\ell$ such that $(\vec{k}_\ell + B_n) \cap U \neq \emptyset$.

Notice that $U \subseteq \bigcup_{\ell \in S} (\vec{k}_\ell + \bo_n) \subseteq U + B_{2n}$. Indeed, if $\vec{i} \in U$, then, since $\{\vec{k}_\ell + B_n\}_{\ell \in \N}$ is a partition of $\Z^d$, there exists (a unique) $\ell^*\in S$ such that $\vec{i} \in (\vec{k}_{\ell^*} + B_n)$. On the other hand, if $\vec{i} \in (\vec{k}_{\ell^*} + B_n)$ for some $\ell^* \in S$, then there exists $\vec{j} \in (\vec{k}_{\ell^*} + B_n) \cap U$, so $\|\vec{i} - \vec{j}\|_\infty \leq 2n$. Therefore, $\vec{i} \in (\vec{j} + B_{2n}) \subseteq U+ B_{2n}$.

Finally, consider $u'$ to be the restriction of $\overline{u}$ to $\bigcup_{\ell \in S} (\vec{k}_\ell + B_n)$. Then, $u'$ is a proper partial coloring and the complement of its support $\{\vec{k}_\ell + B_n\}_{\ell \in \N \setminus S}$ is a (disjoint) union of boxes $B_n$ with perhaps partially defined colorings on its boundary (the restriction of $u'$ to $\partial (\vec{k}_\ell + B_n)$). Then, observing that $B_n$ is just a translation of $[2n+1]^d$, by Proposition \ref{prop:fillable}, we conclude.
\end{proof}

From what has been observed in \cite[Proposition 2.12]{2018factoring}, it follows for $X^d_q$ that FEP with distance $0$ is equivalent to TSSM. Furthermore, it is easy to see that FEP with distance $n$ implies SI with gap $2n$. Considering all this, we have the following result.

\begin{thm}
\label{thm:mixing}
Let $d \ge 1$. Then,
\begin{enumerate}
\item 
For $3 \leq q \leq d+1$,
$X^d_q$ is TM, but not SI.
\item For $d+2 \leq q \leq 2d$, $X^d_q$ is FEP, but not TSSM.
 \item For $2d+1 \leq q$, $X^d_q$ is TSSM.
\end{enumerate}
\end{thm}
In the following, $\mi\in \Z^d$ is called \emph{even} or \emph{odd} if the sum of its coordinates is even or {odd}, respectively. Let $\m e_1, \m e_2, \ldots, \m e_{2d}$ be the unit vectors in $\Z^d$, where $\m e_i=-\m e_{2d-i+1}$ for $1\leq i \leq d$.
\begin{proof}
We split the proof into the three cases in the statement.
\begin{itemize}[leftmargin=20pt,itemsep=0.5em]
\item {\bf Case 1: $3 \leq q \leq d+1$.} In \cite{MR3743365}, it is shown that, for all $d$, $X^d_q$ is topologically mixing if and only if $q \geq 3$. However, $X^d_q$ is not SI if $q \leq d+1$. By Theorem~\ref{thm:frozen-colorings}, we have in fact that there exist frozen $q$-colorings, which is stronger, i.e., if $X^d_q$ has a frozen $q$-coloring, then it cannot be SI: Let $x$ be a frozen $q$-coloring and $y$ any $q$-coloring which differs from $x$ at $\m 0$.
Let $n \ge 1$, $U=\partial B_{n}$ and $V=\{\m 0\}$.
Observe that there does not exists a $q$-coloring $z$ which agrees with $x$ on $U$ and $y$ on $V$, while $\dist(U,V)=n$.

\item {\bf Case 2: $d+2 \leq q \leq 2d$}. By Proposition \ref{prop:list_coloring}, $X^d_q$ is $n$-fillable for $q \geq d+2$. Therefore, by Proposition \ref{prop:fillable_to_FEP}, $X^d_q$ satisfies FEP with distance $n$. 

Fix $n\in \N$ and let $S:=\partial \{m \m e_1: m\in \Z\}$, $U:=\{(0,\ldots, 0)\}$, and $V:=\{n\m e_1\}$. Let $x,y \in X^d_q$ be given by 
$$x_\mi:=\begin{cases}
m+ t\mod (q-2)&\text{ if }\mi=m\m e_1+\m e_t\text{ for }2\leq t \leq 2d-1\text{ and } m\in \Z\\
q-2&\text{ if }\mi\notin \partial \{m \m e_1~:~ m\in \Z\}\text{ is odd}\\
q-1&\text{ if }\mi\notin \partial \{m \m e_1~:~ m\in \Z\}\text{ is even}
\end{cases}$$
and
$$y_\mi:=\begin{cases}
m+ t\mod (q-2)&\text{ if }\mi=m\m e_1+ \m e_t\text{ for }2\leq t \leq 2d-1\text{ and } m\in \Z\\
q-1&\text{ if }\mi\notin \partial \{m \m e_1~:~ m\in \Z\}\text{ is odd}\\
q-2&\text{ if }\mi\notin \partial \{m \m e_1~:~ m\in \Z\}\text{ is even}.
\end{cases}$$

Clearly $x|_S= y|_S$. Suppose that $z\in X^d_q$ is such that $z|_{U\cup S}:=x|_{U\cup S}$. We have that 
\begin{eqnarray*}
z_{(0, \ldots, 0)}=&x_{(0, \ldots, 0)}&= q-1\\
z_{m\m e_1+ \m e_t}=&x_{m\m e_1+ \m e_t}&=m+ t\mod(q-2)
\end{eqnarray*}
for $2\leq t \leq 2d-1$ and $m\in \Z$. Since $q\leq 2d$, it follows that for all $m\in \Z$,
$$\{z_{m\m e_1+ \m e_t}~:~ 2\leq t \leq 2d-1\}=\{0, 1, \ldots, q-3\}$$
and hence $z_{m \m e_1}=x_{m \m e_1} $ for all $m\in \Z$ and $z|_V\neq y|_V$. Since $n$ was arbitrary, we have that $X^d_q$ is not TSSM.

\item {\bf Case 3: $2d+1 \leq q$.} In this case, $X^d_q$ is $0$-fillable and thus satisfies TSSM. \qedhere
\end{itemize}
\end{proof}

\section{Discussion}
\label{section5}
 
\subsection{Gibbs measures and the influence of boundaries} 

One of the key motivations of this paper was to study the influence of a $q$-coloring of the boundary of a box on the colorings inside. Given $n,d,q\in \N$ and $x \in X^d_q$, let
$$X_{x, n, d, q}:= \big\{y\in X^d_q~:~y|_{\Z^d\setminus [n]^d}=x|_{\Z^d\setminus [n]^d}\big\}.$$

If $X^d_q$ is SI, then there exists $n_0\in \N$ such that for all $x\in X^d_q$ and $n\ge n_0$,
\begin{eqnarray*}
1	&	\geq	&	\lim_{n\to \infty} \frac{\log |X_{x,n,d,q}|}{\log |\{\text{$q$-colorings of $[n]^d$}\}|}								\\
	&	\geq 	&	\lim_{n\to \infty} \frac{\log |\{\text{$q$-colorings of $[1+n_0,n-n_0]^d$}\}|}{\log |\{\text{$q$-colorings of $[n]^d$}\}|} = 1.
\end{eqnarray*}

It is not difficult to prove that the limit
$$\lim_{n\to \infty}\frac{1}{n^d} \log{|\{\text{$q$-colorings of $[n]^d$}\}|}$$
exists for all $d,q$ and is referred to as the \emph{entropy (of $X^d_q$)}, denoted by $h_{d,q}$. If the limit 
$$\lim_{n\to \infty}\frac{1}{n^d} {\log |X_{x,n,d,q}|}$$
exists, then it is denoted by $h_{x,d,q}$. 

By Theorem~\ref{thm:mixing}, \eqref{eq:mixing-props} and the calculation above, we have that for $q\geq d+2$, $h_{x,d,q}=h_{d,q}$ for all $x\in X^d_q$. If $q\leq d+1$, then there are frozen $q$-colorings $x\in X_{d,q}$ by Theorem~\ref{thm:frozen-colorings}. For such $x$, $h_{x,d,q}=0$.

\begin{question} Given $q \le d+1$, what is the set of possible values $h_{x,d,q}$ for $x\in X^d_q$? Is it the entire interval $[0,h_{d,q}]$?
\end{question}

This has been established for $q=3$ in \cite{MR2251117} using the ``height function'' formalism, which is missing for other values of $q$.

For $q\geq d+2$, one of the main questions that we would like to address is the following. Let $\mu_{x, n,d, q}$ denote the uniform measure on $X_{x,n,d,q}$.

\begin{question} For what values of $q$ and $d$ does $X^d_q$ have a unique Gibbs measure? In other words, do the measures $\mu_{x,n,d,q}$ converge weakly as $n$ goes to infinity to the same limit for all $x\in X^d_q$?
\end{question}

This has been proved in the case when $q\geq 3.6 d$ \cite{katzmisragamarnik2015} and we suspect that there exists a sequence $q_d$ satisfying $\lim_{d\to \infty} \frac{q_d}{d}=1$ such that it is true when $q\geq q_d$. Note that the existence of frozen colorings preclude the possibility of a unique Gibbs measure, so that Theorem~\ref{thm:frozen-colorings} implies that this does not hold when $q \le d+1$. We also mention that it has recently been shown~\cite{peled2018rigidity} that there are multiple maximal-entropy Gibbs measures when $d \ge Cq^{10}\log^3 q$ for some absolute constant $C>0$.


\subsection{Sampling a uniform $q$-coloring of $[n]^d$}

Suppose that we are to sample a random coloring according to $\mu_{x, n, d, q}$. One way to obtain an approximate such sample is by the Markov Chain Monte Carlo method: construct an ergodic Markov chain on $X_{x,n,d,q}$ whose stationary distribution is $\mu_{x, n, d, q}$, and run it for a long time. A common way to devise such a Markov chain is via the Metropolis--Hastings algorithm for an appropriate set of possible local changes. We mention a couple of such local changes, and address the corresponding ergodicity requirement, namely, whether one can transition between any two elements of $X_{x,n,d,q}$ via the local changes. 

Let us fix $d\geq 2$ and $q\geq 3$ for the following discusssion. We refer to \cite{chandgotia2018pivot} for some more details.

A \emph{boundary pivot move} is a pair $(x,y) \in X_q^d \times X_q^d$ such that they differ at most on a single site. We say that $X^d_q$ has the \emph{boundary pivot property} if for all $x\in X_q^d$, $n\in \N$, and  $y\in X_{x,n, d,q}$ there exists a sequence of boundary pivot moves from $x$ to $y$ contained in $X_{x,n,d,q}$. It is well-known that $X^d_q$ has the boundary pivot property when $q=3$ and it is quite easy to prove it for $q\geq 2d+2$ \cite[Proposition 3.4]{MR3552299}. For $d+2\leq q\leq 2d+1$, a weaker property holds:  A \emph{boundary $N$-pivot move} is a pair $(x,y) \in X_q^d \times  X_q^d$ such that they differ at most on a translate of $[N]^d$. We say that $X^d_q$ has the \emph{generalized boundary pivot property} if there exists $N\in\N$ such that for all $x\in X_q^d$, $n\in \N$, and $ y\in X_{x,n, d,q}$, there exists a sequence of boundary $N$-pivot moves from $x$ to $y$ contained in $X_{x,n,d,q}$. The space $X_{d}^q$ has the generalized boundary pivot property -- this is a consequence of the $n$-fillability property (which holds by Theorem~\ref{thm:list-coloring} and Proposition~\ref{prop:list_coloring}). The proof of this implication follows from the ideas in \cite[Proposition 0.1]{chandgotia2018pivot} (look also at the proof of \cite[Lemma 4.6]{MR3819997} for similar proof).

\begin{question}
For which $q$ and $d$ does $X^d_q$ satisfy the generalized boundary pivot property?
\end{question}

We remark that we do not know of any value of $(q,d)$ for which $X^d_q$ does not satisfy the generalized boundary pivot property. To apply the generalized boundary pivot property, we will still need to be able to sample a uniform coloring on a smaller, but still ostensibly large, box. It will then help to know if another property holds in this case in which such a sampling is not necessary. A \emph{Kempe move} is a pair $(x,y) \in X^d_q \times X^d_q$ such that $y$ is obtained from $x$ by swapping the colors on a bicolor component. We say that $X^d_q$ is \emph{Kempe move connected} if for all $x\in X^d_q$, $n\in \N$, and $y\in X_{x,n,d,q}$, there exists a sequence of Kempe moves from $x$ to $y$ contained in $X_{x,n,d,q}$.

\begin{question}
For which $q$ and $d$ is $X^d_q$ Kempe move connected?
\end{question}
Again, we are not aware of any $(q,d)$ for which $X^d_q$ is not Kempe move connected.

\subsection{Extension of a $q$-coloring of $\partial [n]^d$ to $[n]^d$ for $q\leq d+1$}

It was indicated after the end of the proof of Proposition \ref{prop:list_coloring} that when $3\leq q\leq d+1$, there exist $q$-colorings of $\partial [n]^d$ which do not have an extension into $[n]^d$.

\begin{question} Characterise $q$-colorings of $\partial [n]^d$ which have an extension to $[n]^d$. What is the complexity of determining (in terms of $n$, $q$ and $d$) whether an extension is possible?
\end{question}

For $q=3$, using the ``height functions'' formalism, the complexity is known to be of the order $n\log n$ for $d=2$ (it follows from arguments very similar to those in  \cite{MR3530972}) and $n^{2(d-1)}$ for higher dimensions.


\subsection{Optimizing the parameters of list-colorability.}
We would be interested to determine for which $n$ and $L$ it holds that $[n]^{d}$ is $L$-list-colorable. We have shown that $[n]^d$ is $L$-list-colorable when $n \ge d+2$ and $L = L^d_n$, but we have not tried to optimize $n$ or $L$. We raise two questions, one concerning the optimal $n$, and the other related to improving $L$.

It can be shown (e.g., by Theorem~\ref{thm:digraph-list-coloring}) that $[2]^2$ is $L_2^2$-list-colorable.
%
%
On the other hand, $[2]^3$ is not $L^2_3$-list-colorable. To see this, consider the lists $S:[2]^3\to 2^{\{0,1,2,3\}}$ given by Figure \ref{figure:counter}. The two possible list-colorings of $[2]^2\times \{1\}$ are $\begin{smallmatrix}1&0\\0&2\end{smallmatrix}$ and $\begin{smallmatrix}0&2\\1&0\end{smallmatrix}$, which are incompatible with those of 
$[2]^2\times \{2\}$, that are $\begin{smallmatrix}1&2\\2&3\end{smallmatrix}$ and $\begin{smallmatrix}2&3\\1&2\end{smallmatrix}$.

\begin{figure}[t]
\includegraphics[width=0.5\textwidth]{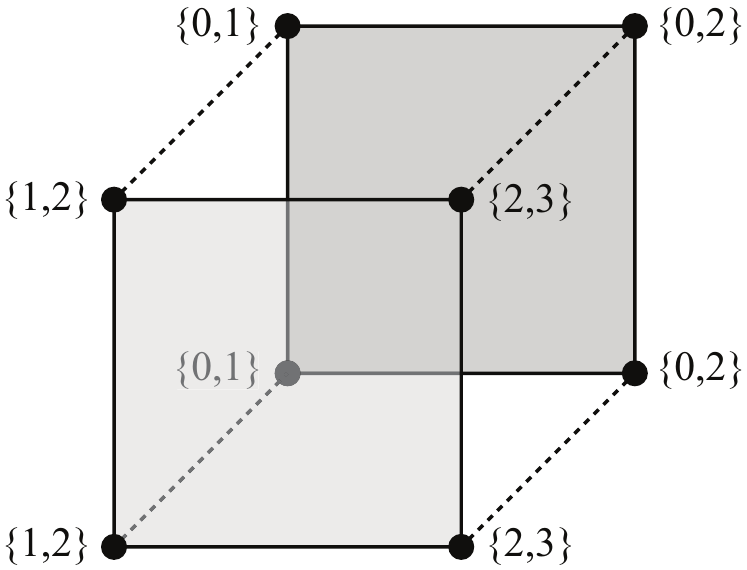}
\caption{The lists $S|_{[2]^2\times \{2\}}$ (front) and $S|_{[2]^2\times \{1\}}$ (back).}
\label{figure:counter}
\end{figure}

Recall that, due to Proposition \ref{prop:mono}, the property of being $L^d_n$-list-colorable is monotone in $n$, so it is natural to ask the following question.

\begin{question}
What is the smallest $n$ for which $[n]^d$ is $L^d_n$-list-colorable?
\end{question}

Regarding improvements to the function $L$, we note that it might be possible to decrease $L_n^d$ pointwise and still find that $[n]^d$ is list-colorable.
For example, in dimension $d=2$, it is not hard to show using Theorem~\ref{thm:digraph-list-coloring}, that $[n]^2$ is $L$-list-colorable for the function $L := \min \{L^2_n,3\}$ (simply orient the external face in a cycle and all other edges in the positive direction; note that $L=L^2_n \equiv 2$ if $n=2$, but otherwise $L \neq L^2_n$).

\begin{question}\label{q:L-list-colorability}
What is the smallest $k$ for which $[n]^d$ is $\min\{L^d_n,k\}$-list-colorable for large enough $n$?
\end{question}

Other functions $L$ which one may wonder about are the constant functions.
We say that $G$ is $k$-list-colorable if it is $L$-list-colorable for the constant function $L \equiv k$.
Though $k$-list-colorability is slightly less related to mixing properties, it is still an interesting combinatorial problem.

\begin{question}
What is the smallest $k$ for which $[n]^d$ is $k$-list-colorable for all $n$?
\end{question}

Using the Lov\'asz local lemma, one may show that $k \ge Cd / \log d$ suffices for some constant $C>0$ (this holds for any triangle-free graph with maximum degree $d$~\cite{johansson1996asymptotic}; see also~\cite{alon1999coloring}).
We remark that this easily leads to a proof that $X^d_q$ is $2$-fillable when $q \ge d + Cd/\log d$, and that this approach to showing fillability can also be useful for proper colorings of other graphs besides $\Z^d$.

\section*{Acknowledgements}

The first author was supported by NSF grant DMS-1855464, ISF grant 281/17, BSF grant 2018267 and the Simons Foundation. The second author was supported by CONICYT/FONDECYT Postdoctorado 3190191 and ERC Starting Grants 678520 and 676970. The third author has been funded by the European Research Council starting grant 678520 (LocalOrder), ISF grant nos. 1289/17, 1702/17 and 1570/17. The fourth author has been funded by European Research Council starting grant 678520 (LocalOrder).

\bibliographystyle{amsplain}
\bibliography{library}

\providecommand{\bysame}{\leavevmode\hbox to3em{\hrulefill}\thinspace}
\providecommand{\MR}{\relax\ifhmode\unskip\space\fi MR }
\providecommand{\MRhref}[2]{%
  \href{http://www.ams.org/mathscinet-getitem?mr=#1}{#2}
}
\providecommand{\href}[2]{#2}
\begin{thebibliography}{10}

\bibitem{alon1999coloring}
Noga Alon, Michael Krivelevich, and Benny Sudakov, \emph{Coloring graphs with
  sparse neighborhoods}, Journal of Combinatorial Theory, Series B \textbf{77}
  (1999), no.~1, 73--82.

\bibitem{alon1992colorings}
Noga Alon and Michael Tarsi, \emph{Colorings and orientations of graphs},
  Combinatorica \textbf{12} (1992), no.~2, 125--134.

\bibitem{MR2645044}
Mike Boyle, Ronnie Pavlov, and Michael Schraudner, \emph{Multidimensional sofic
  shifts without separation and their factors}, Trans. Amer. Math. Soc.
  \textbf{362} (2010), no.~9, 4617--4653. \MR{2645044}

\bibitem{MR3819997}
Raimundo Brice\~{n}o, \emph{The topological strong spatial mixing property and
  new conditions for pressure approximation}, Ergodic Theory Dynam. Systems
  \textbf{38} (2018), no.~5, 1658--1696. \MR{3819997}

\bibitem{2018factoring}
Raimundo Brice\~{n}o, Kevin McGoff, and Ronnie Pavlov, \emph{Factoring onto
  $\mathbb{Z}^d$ subshifts with the finite extension property}, Proceedings of
  the American Mathematical Society \textbf{146} (2018), no.~12, 5129--5140.

\bibitem{MR3702862}
Raimundo Brice\~{n}o and Ronnie Pavlov, \emph{Strong spatial mixing in
  homomorphism spaces}, SIAM J. Discrete Math. \textbf{31} (2017), no.~3,
  2110--2137. \MR{3702862}

\bibitem{briceno2019long}
Raimundo Brice{\~n}o, Andrei Bulatov, Victor Dalmau, and Benoit Larose,
  \emph{Long range actions, connectedness, and dismantlability in relational
  structures}, arXiv preprint arXiv:1901.04398 (2019).

\bibitem{brightwell2000gibbs}
Graham~R. Brightwell and Peter Winkler, \emph{Gibbs measures and dismantlable
  graphs}, J. Combin. Theory Ser. B \textbf{78} (2000), no.~1, 141--166.
  \MR{1737630 (2001e:05129)}

\bibitem{brightwell2002random}
Graham~R Brightwell and Peter Winkler, \emph{Random colorings of a {C}ayley
  tree}, Contemporary combinatorics \textbf{10} (2002), 247--276.

\bibitem{burtonsteiffnonuniquesft}
Robert Burton and Jeffrey~E. Steif, \emph{Non-uniqueness of measures of maximal
  entropy for subshifts of finite type}, Ergodic Theory Dynam. Systems
  \textbf{14} (1994), no.~2, 213--235. \MR{1279469 (95f:28023)}

\bibitem{chandgotia2018pivot}
Nishant Chandgotia, \emph{A short note on the pivot property},
  {http://math.huji.ac.il/$\sim$nishant/Research\_files/Notes\_files/Pivot.pdf}
  (2018).

\bibitem{MR3743365}
Nishant Chandgotia and Brian Marcus, \emph{Mixing properties for hom-shifts and
  the distance between walks on associated graphs}, Pacific J. Math.
  \textbf{294} (2018), no.~1, 41--69. \MR{3743365}

\bibitem{MR3552299}
Nishant Chandgotia and Tom Meyerovitch, \emph{Markov random fields, {M}arkov
  cocycles and the 3-colored chessboard}, Israel J. Math. \textbf{215} (2016),
  no.~2, 909--964. \MR{3552299}

\bibitem{katzmisragamarnik2015}
David Gamarnik, Dmitriy Katz, and Sidhant Misra, \emph{Strong spatial mixing of
  list coloring of graphs}, Random Structures Algorithms \textbf{46} (2015),
  no.~4, 599--613. \MR{3346458}

\bibitem{johansson1996asymptotic}
Anders Johansson, \emph{Asymptotic choice number for triangle free graphs},
  Tech. report, DIMACS technical report, 1996.

\bibitem{jonasson2002uniqueness}
Johan Jonasson, \emph{Uniqueness of uniform random colorings of regular trees},
  Statistics \& Probability Letters \textbf{57} (2002), no.~3, 243--248.

\bibitem{2015integral}
Brian Marcus and Ronnie Pavlov, \emph{An integral representation for
  topological pressure in terms of conditional probabilities}, Israel Journal
  of Mathematics \textbf{207} (2015), no.~1, 395--433.

\bibitem{MR3530972}
Igor Pak, Adam Sheffer, and Martin Tassy, \emph{Fast domino tileability},
  Discrete Comput. Geom. \textbf{56} (2016), no.~2, 377--394. \MR{3530972}

\bibitem{peled2018rigidity}
Ron Peled and Yinon Spinka, \emph{Rigidity of proper colorings of {${\mathbb
  Z}^d$}}, arXiv preprint arXiv:1808.03597 (2018).

\bibitem{MR1359979}
Klaus Schmidt, \emph{The cohomology of higher-dimensional shifts of finite
  type}, Pacific J. Math. \textbf{170} (1995), no.~1, 237--269. \MR{1359979}

\bibitem{MR2251117}
Scott Sheffield, \emph{Random surfaces}, Ast\'erisque (2005), no.~304, vi+175.
  \MR{2251117}

\end{thebibliography}

\end{document}